\documentclass[12pt]{article}

\usepackage{amsfonts}
\usepackage{amsmath}
\usepackage{amssymb}
\usepackage{amsthm}
\usepackage{graphicx}
\usepackage{subfigure}

\def\ra{\rightarrow}

\def \Z {\mathbb Z}

\def\P{\mathbb P}

\def\ep{\varepsilon}

\def\La{\Lambda}

\def\Ph1{P^{(h_1)}}
\def\Ph2{P^{(h_2)}}

\def\b0{{\bf 0}}

\def\bnu{\bar\nu}

\def\1lbnu{{\bnu}_{\lambda_1}}
\def\2lbnu{{\bnu}_{\lambda_2}}

\newtheorem{thm}{Theorem}[section]

\newtheorem{lem}[thm]{Lemma}

\theoremstyle{plain}

\begin{document}

\title{A percolation process on the square lattice where large finite clusters are frozen} 

\author{Jacob van den Berg\footnote{CWI and VU University, Amsterdam; J.van.den.Berg@cwi.nl},
Bernardo N.B. de Lima\footnote{Universidade Federal de Minas Gerais, Belo Horizonte; bnblima@mat.ufmg.br}
and Pierre Nolin\footnote{Courant Institute, NYU, New York; nolin@cims.nyu.edu}  
}

\date{}
\maketitle

\begin{abstract}
In \cite{Al}, Aldous constructed a growth process for the binary tree where clusters freeze as soon as they become
infinite. It was pointed out by Benjamini and Schramm that such a process does not exist for the square lattice.

This motivated us to investigate the modified process on the square lattice, where clusters freeze
as soon as they have diameter larger than or equal to $N$, the parameter of the model. The non-existence result,
mentioned
above, raises the question if the $N-$parameter model shows some `anomalous' behaviour as $N \ra \infty$.
For instance, if one looks at the cluster of a given vertex, does,
as $N \ra \infty$, the probability that it eventually freezes go to 1?
Does this probability go to $0$? More generally, what can be said about the size
of a final cluster? We give a partial answer to some of such questions.

\end{abstract}

\section{Introduction and statement of the main result}
Let $S$ denote the square lattice. The vertices of this lattice are the elements of $\Z^2$, and each vertex $v$ has an
edge to each of the four vertices $v + (i,j)$, \,\, $|i| + |j| = 1$. Let $E$ denote the set of edges of $S$.
The norm $|v|$ of a vertex $v = (v_1, v_2)$ is defined as $\max(|v_1|, |v_2|)$, and the distance between
two vertices $v$ and $w$ is defined as $|v - w|$.
The diameter of a set $W \subset \Z^2$ is defined as $\sup\{|v-w| \, : \, v, w \in W\}$. By the diameter of
a subgraph $G$ of the square lattice, we mean the diameter of the set of vertices of $G$. 

To each edge $e \in E$ we
assign a value $\tau_e$, where the $\tau_e, e \in E$ are independent random variables, uniformly distributed on
the interval $(0,1)$. 
At time $0$ all edges are closed. If each edge $e$ became open at time $\tau_e$ (and remained open after that time),
the configuration of open and closed edges at time $t$ would simply be a typical configuration for an ordinary percolation
model with parameter $t$. In particular, the open cluster of a given vertex, say $0$, would initially consist of $0$
only, remain finite up to some (random) time $t > 1/2$, and eventually (at time $1$) be the entire lattice.

However, in the process we study, each open cluster `freezes' as soon as it has diameter larger than or equal to $N$, the
parameter of the process. Here `freezes' means that the external edges of the cluster remain closed forever. 

In other words, in this process initially all edges are closed, and an edge $e$ becomes open at time $\tau_e$,
unless at least one endpoint of $e$ already belongs to an open cluster with diameter $\geq N$ (in which case
$e$ remains closed forever).

We are interested in the sizes (diameters) of the final open clusters (i.e. the open clusters at time $1$), 
for large values of the parameter $N$. Note that the rules of the process immediately imply
that the diameter of an open cluster cannot be more than $2 N-1$. (The diameter $2 N -1$ can be obtained
by merging two well-chosen clusters of diameter $N-1$ each).

Motivation comes from a paper by Aldous \cite{Al}, who introduced (as an `interpretation' of ideas concerning
gel formation in \cite{St}) a growth process with similar rules as above, but where
clusters freeze as soon as they become infinite. We will refer to that model as the $\infty$-parameter frozen
model. Aldous made a rigorous construction of such a process for the
binary tree (and proved several interesting properties). However, Benjamini and Schramm ((1999), private communication via 
D. Aldous) showed that such a process does not exist for the square lattice (see the discussion in \cite{BeTo}, Section 3).

It follows from standard arguments that for each finite $N$, the
$N$-parameter frozen percolation model on $S$ (and, more generally, on $\Z^d$) does exist.
(See Sections 4.1 and 4.2 in \cite{Br}, where also some
exact computations for $d = 1$ are shown).
It is natural to ask if
the above mentioned non-existence result for the $\infty$-parameter model on $S$ is,
in some sense, reflected in the asymptotic behaviour of the $N$-parameter system as $N \ra \infty$.
In particular, the following questions arise, where we use the notation $C^{(N)}$ for the open cluster of the origin at time
$1$ in the $N$-parameter model, and where a cluster is called a {\em giant cluster} if its diameter is at least $N$.

\begin{itemize}
\item (1.) Do, eventually, the giant clusters cover the entire lattice? More precisely, 
$$ \text{ Does } \P\left(C^{(N)} \text{ has diameter } \geq N\right) \ra 1, \,\, \text{ as } N \ra \infty \, \text{?} $$

\item (2.) Do, eventually, the giant clusters cover a neglible portion of the lattice? More precisely,
$$ \text{ Does } \P\left(C^{(N)} \text{ has diameter } \geq N\right) \ra 0, \,\, \text{ as } N \ra \infty \, \text{?} $$

\item (3.) If the answer to question (1) is negative, what can be said, for large $N$, about the diameters
of the non-giant clusters?
\end{itemize}

Note that if the final cluster of $0$ has diameter $k$, there is a vertex at distance $\leq k+1$ from $0$ which
belongs to a giant cluster. Hence, if the answer to question (2) is positive, then, for every $k$, the probability that
$C^{(N)}$ has diameter $k$ goes to $0$ as $N \ra \infty$. 

Theorem \ref{mainthm}, below, gives a negative answer to question (1) and a partial answer
to question (3).
The proof is given in the next section.

\begin{thm}\label{mainthm}
Let, as before, $C^{(N)}$ denote the open cluster of the origin at time $1$ for the $N$-parameter frozen percolation model on the
square lattice. \\
For all $0 < a < b < 1$,
$$\liminf_{N \rightarrow \infty} P\left(C^{(N)} \mbox{ has diameter } \in (a N, b N)\right) > 0.$$
\end{thm}

\section{Proof of Theorem \ref{mainthm}}

Deviating somewhat from the standard percolation notation, we define
(for a positive even integer $k$) $B(k)$ as the box $[-k/2, k/2]^2$ in the square lattice.

Let, as before, $N$ denote the parameter in the frozen percolation process.

Let $0< a < b < 1$ be given, and choose $c \in (a,b)$. Next, take $l$ such that 

\begin{equation}\label{eqi}
l + (b-c)/2 < 1 < l + (b+c)/2.
\end{equation}

By the first inequality in \eqref{eqi} we can choose $0 < \varepsilon < 1$ so small that also

\begin{equation}\label{eqii}
l + (b-c)/2 + \varepsilon < 1.
\end{equation}

(Later we possibly make $\varepsilon$ even smaller to satisfy additional conditions).

\begin{figure}
\begin{center}

\subfigure[The different annuli and rectangles used.]{\includegraphics[width=\textwidth]{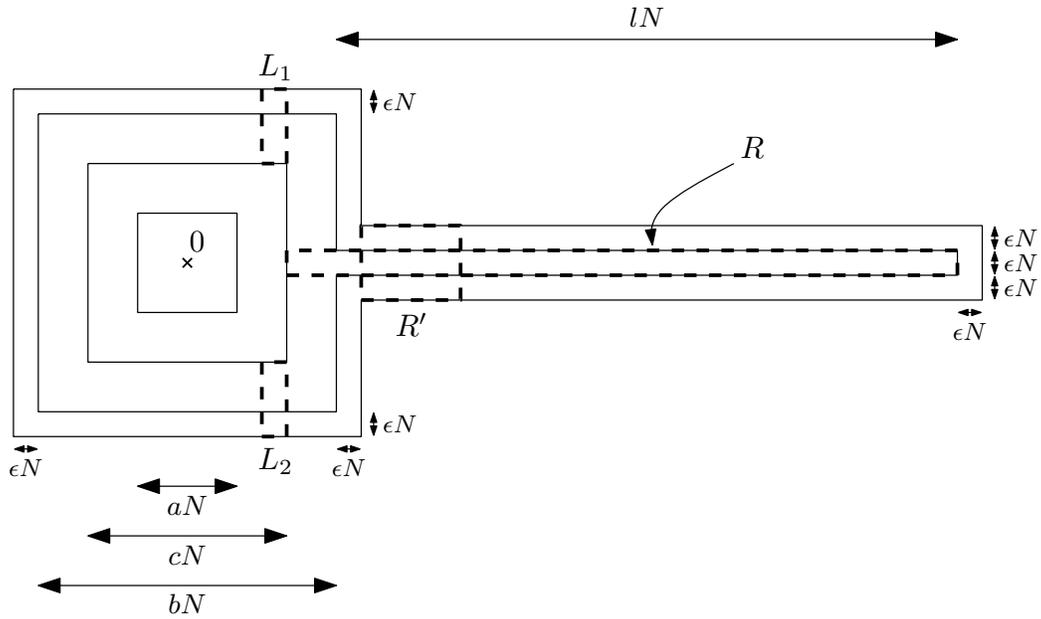}}
\subfigure[Open and closed paths involved in the proof.]{\includegraphics[width=\textwidth]{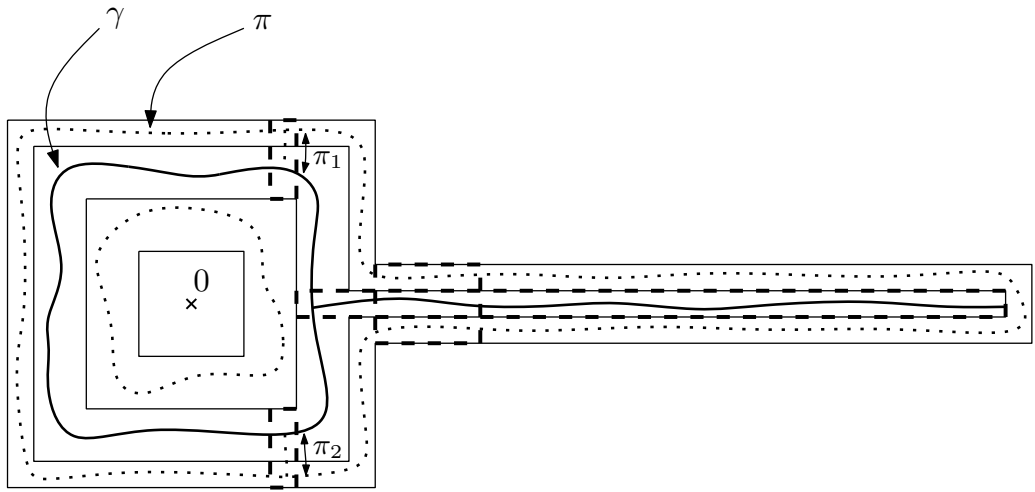}}

\caption{\label{fig:main}Construction used for the proof of Theorem \ref{mainthm}. The open paths are represented with solid
lines, and the closed paths with dotted lines.}
\end{center}
\end{figure}

Let $R$ be the rectangular box of length $(l + (b-c)/2)N$ and width $\ep N$ of which the west side is a central subsegment
of the east side of $B(c N)$. Let $\La$ be the union of $B(b N)$ and $R$. So $\La$ is a $b N \times b N$ square
from which a $l N \times \ep N$ rectangle sticks out to the right (see Figure \ref{fig:main}).
Further, let $\La'$ be the set of all points at distance $\leq \ep N$ from $\La$.
So $\La'$ is the disjoint union of $B((b + 2 \ep)N)$ and a rectangle of width $3 \ep N$ and
length $l N$. Let $R'$ be the leftmost part of length $4 \ep N$ of that rectangle.

Finally, let $L_1$ (respectively, $L_2$) be the rectangle of which the south (resp. north) side
is the rightmost segment of length
$\ep N$ of the north (resp. south) side of $B(c N)$, and the north (resp. south) side
is a part of the boundary of $\La'$.

The first part of the proof of the theorem is a deterministic Lemma.
Call an edge $e$ $t$-open if $\tau_e < t$. A path, or more generally a set of edges,
is called $t$-open if every edge of that
set is $t$-open. The terminology $t$-closed is defined in a completely similar way. 

\begin{lem}\label{lem1}
Suppose that there is a $\tau \in (0, 1/2)$
such that each of the following, (i) - (vi) below, holds:

\begin{itemize}
\item (i) $\exists \tau$-open circuit $\gamma$ in the annulus $B(b N) \setminus B(c N)$.
\item (ii) $\exists \frac{1}{2}$-closed dual circuit in the annulus $B(c N) \setminus B(a N)$.
\item (iii) $\exists$ a $\frac{1}{2}$-closed dual circuit $\pi$ in the annulus $\Lambda' \setminus
\Lambda$.
\item (iv) $\exists$ $\frac{1}{2}$-closed dual paths $\pi_1$ and $\pi_2$ in $L_1$, respectively $L_2$,
`connecting' $\gamma$ and $\pi$.
\item (v) $\exists$ a $\frac{1}{2}$-open path in $R$ from the right side of $R$ to $\gamma$.
\item (vi) There is no $\tau$-open horizontal crossing of the area of $R'$ bounded by the two segments
of $\pi$.
\end{itemize}

Then the final cluster of $0$ in the $N$-parameter frozen percolation process has diameter $\in [a N, b N]$.
\end{lem}

\begin{proof}
What are the consequences of (i) - (vi) for the $N$-parameter frozen percolation process?
Note
that before time $\tau$ nothing in the interior of $\pi$ is frozen yet, simply because, roughly
speaking, by (vi) this region is split in two subregions that are too small to freeze. Consequently,
every edge $e$ in the interior of $\pi$ with $\tau_e < \tau$ is indeed open at time $\tau$ in the frozen
process. 
 In particular the circuit $\gamma$ (see (i) above) is open at time $\tau$. Because of this, and by (iv), there can be no 
cluster in the interior of $\pi$ that does not contain $\gamma$ and freezes before time $1/2$. On the other hand, something
in the interior of $\pi$ does freeze before time $1/2$ because, by (i) and (v), there is a $1/2$ open connected
component in the interior of $\pi$ with diameter $\geq (l + (b-c)/2 + c) N$, which by \eqref{eqi} is larger than $N$.
Hence the open cluster containing $\gamma$ is frozen before time $1/2$. Further, because of this and by (ii),
it follows that at time $1/2$ the open cluster of $0$ is not frozen but is surrounded by a circuit of diameter
$< b N$ which is part of a frozen cluster. Hence the diameter of the open cluster of $0$ will remain smaller than $b N$ forever.
On the other hand, every edge in the interior of the $1/2$-closed circuit in (ii) will eventually become open.
Hence the diameter of the final open cluster of $0$ is between $a N$ and $b N$.
This completes the proof of Lemma \ref{lem1}.

\end{proof}

To complete the proof of Theorem \ref{mainthm}, we will show that there exists a $\tau \in (0,1/2)$ for which
the probability that there is a circuit $\gamma$ in the annulus $B(b N) \setminus B(c N)$ and a dual circuit
$\pi$ in the annulus $\Lambda' \setminus \Lambda$ such that each of the events (i) - (vi) in Lemma \ref{lem1} holds
is bounded away from $0$ as
$N \rightarrow \infty$. 
From now on, $E_{(i)}$, $E_{(ii)}$  and $E_{(iii)}$ will denote the event described in (i), respectively (ii) and (iii),
in Lemma \ref{lem1}. Moreover, for a given (deterministic) circuits $\gamma$ in the annulus $B(b N) \setminus B(c N)$
and a given dual circuit $\pi$ in the annulus $\Lambda' \setminus \Lambda$, we define the following events: 

\smallskip\noindent
$E_{(i)}(\gamma)$ is the
event that $\gamma$ is the narrowest $\tau$-open circuit in $B(b N) \setminus B(c N)$; $E_{(iii)}(\pi)$ 
is the event that $\pi$ is the widest $1/2$-closed dual circuit in $\Lambda' \setminus \Lambda$;
$E_{(iv)}(\gamma,\pi)$ is the event that (iv) holds; 
$E_{(v)}(\gamma)$ is the event that (v) holds; 
$E_{(vi)}(\pi)$ is the event that (vi) holds. 

\smallskip\noindent
Further, $C_{(i)}$, $C_{(ii)}$ etc. will denote
strictly positive constants that may depend on $a$, $b$, $c$, and $\ep$, but not on $N$.

Now, let $\alpha$ ($= \alpha(N)$) denote the probability that there is
a $1/2$-open horizontal crossing of the rectangle $R$. Note that, by the well-known RSW results in
percolation (see e.g. \cite{Gr}, Section 11.7), $\alpha(N)$ is bounded away from
$0$ as $N \rightarrow \infty$. Let $\tau$ ($= \tau(N)$) be such that 

\begin{equation}\label{tau-def}
P(\exists \tau \mbox{-open horizontal crossing of } R') = \alpha/2.
\end{equation}
Note that for all sufficiently large $N$ such a $\tau$ exists, is unique and smaller than $1/2$, because the probability that
there is a $t$-open horizontal crossing of $R'$ is
obviously continuous and strictly increasing in $t$, is $0$ for $t = 0$, and at least $\alpha$ for $t = 1/2$.

Now, again by RSW, there is a strictly increasing function $f : [0,1] \rightarrow [0,1]$ (which depends only
on $b$, $c$, and $\ep$, but not on $N$ or $t$) such that $f(0) = 0$, $f(1) = 1$ and
\begin{align*} 
P( \exists t \mbox{-open circuit in } & B(b N) \setminus B(c N))\\
& \geq f\left(P(\exists t \mbox{-open horizontal crossing of } R')\right).
\end{align*}

By this, the definition of $E_{(i)}$,  and \eqref{tau-def},
there is a positive constant $C_{(i)}$ such that for all sufficiently large $N$ 
\begin{equation} \label{boundi}
P(E_{(i)}) \geq C_{(i)}.
\end{equation}

Also by RSW, there is a positive constant $C_{(iii)}$ such that for all sufficiently large $N$

\begin{equation} \label{boundiii}
P(E_{(iii)}) \geq C_{(iii)}.
\end{equation}

Now, for given $\gamma$ and $\pi$, we condition on the event $E_{(i)}(\gamma)\cap E_{(iii)}(\pi)$ defined above.

Note that the event $E_{(ii)}$ is independent of this event and (by RSW) has,
for all sufficiently large $N$, 
probability larger than some positive constant $C_{(ii)}$.
Also note that, again by RSW, the conditional probability of $E_{(iv)}(\gamma, \pi)$ 
is (for all sufficiently large $N$) bounded from below by some positive constant $C_{(iv)}$.
Further, $E_{(v)}(\gamma)$ is independent of the event we condition on, and clearly (by the definition of $\alpha$),
has conditional probability $\geq \alpha$.
It is also clear that $E_{(vi)}(\pi)$ is independent of the event we condition on, and that (using the choice of $\tau$),
for all sufficiently large $N$,
its complement has probability at most $\alpha/2$. 
Finally, it is easy to see that the events $E_{(ii)}$, $E_{(iv)}(\gamma, \pi)$, and $E_{(v)}(\gamma) \cap E_{(vi)}(\pi)$
are conditionally independent.

Combining the above facts with Lemma \ref{lem1} we get that the probability that
the final cluster of $0$ in the $N$-parameter frozen percolation process has diameter $\in [a N, b N]$
is larger than or equal to

\begin{equation}\label{finsum}
\sum_{\gamma, \pi} P(E_{(i)}(\gamma)\cap E_{(iii)}(\pi)) \, C_{(ii)} \, C_{(iv)} P(E_{(v)}(\gamma) \cap E_{(vi)}(\pi)).
\end{equation}

Since, by the choice of $\alpha$ and by \eqref{tau-def}, 
$$P(E_{(v)}(\gamma) \cap E_{(vi)}(\pi)) \geq P(E_{(v)}(\gamma)) - P\left(E_{(vi)}^c(\pi)\right) \geq
\alpha - \alpha/2 = \alpha/2,$$
the summation \eqref{finsum} is larger than or equal to

$$\frac{\alpha}{2} \, C_{(ii)} \, C_{(iv)} \sum_{\gamma, \pi} P\left(E_{(i)}(\gamma)\cap E_{(iii)}(\pi)\right),$$
which by \eqref{boundi} and \eqref{boundiii} is larger than or equal to 
$$\frac{\alpha}{2} C_{(ii)} \, C_{(iv)} \, C_{(i)} C_{(iii)}.$$
This completes the proof of Theorem \ref{mainthm}. $\square$

\section*{Acknowledgments}
B.N.B.L. is partially supported by CNPq
and FAPEMIG (Programa Pesquisador Mineiro). P.N.'s research was supported in part by the NSF grant OISE-07-30136.
B.N.B.L. and P.N. would also like to thank CWI for its hospitality during multiple visits.

\end{document}